\documentclass[12pt]{article}
 \usepackage{amsmath}
\usepackage{amsfonts}
\usepackage{tikz}
\usepackage[title]{appendix}
\usepackage{subfig}
\usepackage{array}
\usepackage{blkarray}
\usepackage{tabu}
\usepackage{nth}
\usepackage[
  separate-uncertainty = true,
  multi-part-units = repeat
]{siunitx}

\usetikzlibrary{arrows}
\usepackage{epsfig,}
\usepackage{epsfig}
\tikzset{
    vertex/.style = {
        circle,
        draw,
        outer sep = 3pt,
        inner sep = 3pt,
    },edge/.style = {->,> = latex'}
}

\usepackage{amssymb}
\usepackage{enumitem}% http://ctan.org/pkg/enumitem
\usepackage{amsthm}% http://ctan.org/pkg/amsthm
\usepackage{dsfont}
\def\diag{\mathop{\rm diag}}

\def\rank{\mathop{\rm rank}}
\newcommand{\rr}{\mathbb{R}}
\newcommand{\1}{\mathbf{1}}
\newcommand{\0}{\mathbf{0}}

\def\det{{\rm det}}

\def\cir{{\rm Circ}}

\newcommand{\I}{({\rm I}_1)}
\newcommand{\J}{({\rm I}_2)}
\newcommand{\K}{({\rm I}_3)}
\newcommand{\KK}{({\rm I}_4)}
\newcommand{\KKK}{({\rm I}_5)}
\def\L{\widetilde{L}}
\def\D{\widetilde{D}}

\def\v{\widetilde{v}}

\newtheorem{theorem}{Theorem}

\newtheorem{lemma}{Lemma}
\newtheorem{definition}{Definition}

\newcommand{\al}{\alpha}
\begin{document}
\begin{center}
\begin{large}
On distance matrices of wheel graphs with odd number of vertices 
\end{large}
\end{center}
\begin{center}
R. Balaji, R.B. Bapat and Shivani Goel \\
(In memory of Miroslav Fiedler)\\
\today
\end{center}

\begin{abstract}
Let $W_n$ denote the wheel graph having $n$-vertices. 
If $i$ and $j$ are any two vertices of $W_n$, define 
\[d_{ij}:=
    \begin{cases}
    0 & \mbox{if~}i=j\\
    1 & \mbox{if~} i ~\mbox{and}~ j~ \mbox{are adjacent} \\
    2 & \mbox{else}.
    \end{cases}\]
   Let $D$ be the $n \times n$ matrix with $(i,j)^{\rm th}$ entry equal to $d_{ij}$. The matrix $D$ is called the distance matrix of $W_n$.
 Suppose $n \geq 5$ is an odd integer. In this paper, we
deduce a formula to compute the Moore-Penrose inverse of $D$. 
More precisely, we obtain an $n\times n$ matrix $\L$ and  a rank one matrix $ww'$ such that 
\[D^{\dag}= -\frac{1}{2} \L+\frac{4}{n-1}ww'.\]
Here, $\L$ is positive semidefinite, $\rank(\L)=n-2$ and all row sums are equal to zero. 

\end{abstract}


{\bf AMS CLASSIFICATION.} 05C50

\section{Introduction}
Let $G$ be a connected graph with vertex set $V:=\{1,\dotsc,n\}$.  
Since $G$ is connected, any two vertices $i$ and $j$ in $V$ are now connected by a path in $G$. Let the 
minimum length of all such paths be denoted by $d_{ij}$.
 The distance matrix of $G$ is then the $n \times n$ symmetric matrix with $(i,j)^{\rm th}$ off-diagonal entry equal to $d_{ij}$ and all diagonal entries equal
 to zero. Distance matrices of connected graphs have several interesting properties and have applications in various fields like
 data communication, chemistry and biology. Distance matrices have a wide literature. For a comprehensive 
 introduction, we refer to the survey article  
\cite{must} and the monograph \cite{bapat} and \cite{fiedler_2011}.
There are several interesting problems on distance matrices. One of them is the following: If $G$ is a connected graph and $D$ is the distance matrix of $G$, deduce a formula to compute the determinant and the inverse of $D$. 
This problem originates from a well-known result of 
Graham and Lov\'asz \cite{Graham}. To introduce this result, we need to recall the notion of the Laplacian matrix of $G$.
Define $S:=\diag(s_1,\dotsc,s_n)$, where $s_i$ is the degree of
the vertex $i$ in $G$.
Suppose $A$ is the adjacency matrix
of $G$. Then
the matrix $M:=S-A$ is called the Laplacian matrix of $G$ with the following properties:
\begin{enumerate}
\item[(M1)] $M$ is positive semidefinite.
\item[(M2)] All row sums of $M$ are zero.
\item[(M3)] $\rank(M)=n-1$.
\end{enumerate}
Suppose $E$ is the distance matrix of a tree with $n$-vertices. According to Graham and Lov\'asz \cite{Graham},
\begin{equation} \label{gl}
E^{-1}=-\frac{1}{2}L + \frac{1}{2(n-1)}\tau \tau', 
\end{equation}
where $L$ is the Laplacian matrix of the tree and $\tau=(2-\delta_1,\dotsc,2-\delta_n)'$  with $\delta_i$ equal to the degree of the vertex $i$.
The remarkable feature of this formula is that the inverse can be expressed just by using the adjacency matrix and the vertex degrees of the tree.
A question that arises now naturally is how to generalize formula (\ref{gl}) to connected graphs other than trees.
In the case of trees, there is an elegant identity that connects the Laplacian with the distance matrix. If $E=[e_{ij}]$ and
$L^{\dag}=[\beta_{ij}]$, then 
\begin{equation} \label{ldid}
e_{ij}=\beta_{ii}+\beta_{jj}-2 \beta_{ij},
\end{equation}
where $L^{\dag}$ is the Moore-Penrose inverse of the Laplacian $L$.
All the known proofs for $(\ref{gl})$ rely on the relation $(\ref{ldid})$ either directly or indirectly and the properties (M1), (M2) and (M3) of the Laplacian.
If the connected graph is not a tree, then the identity (\ref{ldid}) does not hold and hence in general it is very difficult to get an elegant formula
similar to $(\ref{gl})$. 
However, for some special cases like weighted trees, complete graphs, complete bipartite graphs and wheel graphs with even number of vertices, there are formula 
in the spirit of $(\ref{gl})$: see \cite{Balaji, bapat_kirk, sivasu}.

Let $W_n$ be the wheel graph having $n$-vertices. In this paper, we assume $n$ is an odd integer.
Suppose $D$ is the distance matrix of $W_n$.
Define a vector $d := (0,1,-1,1,-1,\dotsc,-1)' \in \rr^n$. Then $Dd=\0$. So, $\det(D) = 0$. 
We now deduce a formula to compute the Moore-Penrose inverse of $D$ which is similar to 
(\ref{gl}). Precisely, we
obtain a matrix $\L$ and a rank one matrix $ww'$ such that 
\[D^{\dag}= -\frac{1}{2} \L+\frac{4}{n-1}ww',\]
where $\L$ is positive semidefinite, $\rank(\L)=n-2$ and all row sums are equal to zero. We also show that if $\L^{\dag}=[\theta_{ij}]$, then
\[d_{ij}=\theta_{ii}+\theta_{jj}-2 \theta_{ij} .\]

\section{Notation and conventions}
\begin{itemize}
\item The notation $n$ will {\it always} denote an odd positive integer which is at least $5$ and $W_n$ will stand for the wheel graph with
$n$ number of vertices.
The center of $W_n$ is labelled $w_1$. All vertices other than $w_1$  lie in a cycle of length $n-1$. We label these vertices by $w_2,w_3,\dotsc,w_{n-1}$ such that $(w_i,w_{i+1})$ is an edge. 
For example, see Figure \ref{fig_W7}.  
\begin{figure}
\centering
\begin{tikzpicture}[shorten >=1pt, auto, node distance=3cm, ultra thick,
   node_style/.style={circle,draw=black,fill=white !20!,font=\sffamily\Large\bfseries},
   edge_style/.style={draw=black, ultra thick}]
\node[vertex] (1) at  (0,0) {$w_1$};
\node[vertex] (2) at  (1.5,-2) {$w_2$};
\node[vertex] (3) at  (2.5,0) {$w_3$}; 
\node[vertex] (4) at  (1.5,2) {$w_4$};  
\node[vertex] (5) at  (-1.5,2) {$w_5$};  
\node[vertex] (6) at  (-2.5,0) {$w_6$};  
\node[vertex] (7) at  (-1.5,-2) {$w_7$};  
%\node[vertex] (5) at  (0,-2) {$w_$};  
\draw  (1) to (2);
\draw  (1) to (3);
\draw  (1) to (4);
\draw  (1) to (5);
\draw  (1) to (6);
%\draw (6) to (2);
\draw  (1) to (7);
\draw  (2) to (3);
\draw  (3) to (4);
\draw  (4) to (5);
\draw  (5) to (6);
\draw  (6) to (7);
\draw  (7) to (2);
\end{tikzpicture}
\caption{Wheel graph $W_7$} \label{fig_W7}
\end{figure}
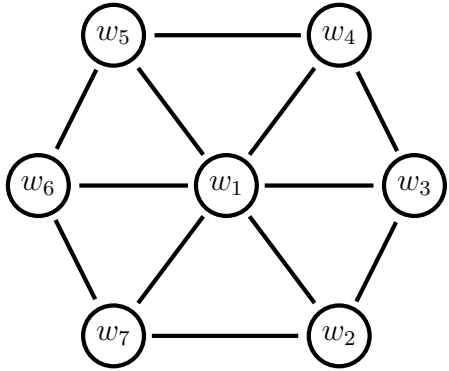
Since any other labelling of $W_n$ leads to a distance matrix which is permutation similar to $D$, without loss of generality, we fix this labelling.

\item All vectors are assumed to be column vectors unless stated otherwise. The identity matrix of order $n$ is denoted by $I$. If $k<n$, we use $I_k$ to denote the identity matrix of order $k$. 

\item We denote the vector of all ones in $\rr^{n-1}$ by $\1$ and the $(n-1) \times (n-1)$ matrix of all ones by $J$. If $\nu \neq n-1$, we use the notation $\1_\nu$ to denote the vector of all ones in $\rr^{\nu}$ and $J_\nu$ to denote the $\nu \times \nu$ matrix of all ones. As usual, we use $0$ to denote the scalar zero. To denote the zero vector (row/column), we use the notation $\0$. A matrix with more than one row/column and having all entries equal to zero is denoted by $O$.
If $(x_1,\dotsc,x_k)$ is a row vector, then $\cir(x_1,\dotsc,x_k)$ will be the circulant matrix with first row equal to $(x_1,\dotsc,x_k)$.

\item We reserve the letter $u$ to denote the row vector $(0,1,2,\dotsc,2,1)$ with $n-1$ components.
The distance matrix of $W_n$ now has the form
\begin{equation}\label{E_block_form_D}
    D = \left[ \begin{array}{cc}
        0 & \1' \\
        \1 & \widetilde{D}
    \end{array}\right],
\end{equation}
where $\widetilde{D}=\cir(u)$.  
We record the equation 
\begin{equation} \label{fd}
\D \1=2(n-3) \1
\end{equation}
for later use. 
An $n \times n$ matrix $A = [a_{ij}]$ is an Euclidean distance matrix if there exist $x^1,x^2,\dotsc,x^n \in \rr^r$ such that $a_{ij} = (x^i-x^j)'(x^i-x^j)$. By Theorem 12 in \cite{Jakli}, it follows that $D$ is an Euclidean distance matrix.

\item Let $x=(x_1,\dotsc,x_{n-1})$. We say that $x$ follows symmetry with respect to the $(\frac{n+1}{2})^{\rm th}$ coordinate in its last $n-2$ coordinates 
if $x$ has the form \[(x_1,x_2,\dotsc,x_\frac{n-1}{2},x_{\frac{n+1}{2}},x_\frac{n-1}{2},\dotsc,x_3,x_2),\] or
equivalently, $x$ satisfies the equations
\[x_i=x_{n+1-i}~~\mbox{for all}~ i=2,3,\dotsc,n-1. \]  
We define \[\Delta:=\{(x_1,\dotsc,x_{n-1}): x_i=x_{n+1-i}~\mbox{for all} ~i=2,3,\dotsc,n-1\}. \]

\item We fix $m$ to denote $\frac{n-1}{2}$. For each $k \in \{1,2,\dotsc,m\}$, define $ (c_{1}^{k},\dotsc,c_{n-1}^{k}) $ by
\begin{equation*}
    c^k_j := \begin{cases}
    1 & j=k+1~\mbox{or}~j=n-k \\ 
    0 &\mbox{otherwise}.
    \end{cases}
\end{equation*}
Let $c^k:=(c_1^k,\dotsc,c_{n-1}^{k})$ and
$C_k:=\cir(c^k)$. By an easy verification, 
\begin{equation} \label{ckd}
C_k \D=\cir(c^{k} \D).
\end{equation}
We  shall say that $c^1,\dotsc,c^m$ are special vectors for $W_n$ and $C_1, C_2,\dotsc,C_m$ are special matrices for $W_n$. Each $C_k$ is symmetric.
For $i \in \{1,2,\dotsc,n-1\}$, define
\begin{equation*}
    v_i = \begin{cases}
    ~~1&~\mbox{if}~i~\mbox{is odd}\\
    -1&~\mbox{if}~i~\mbox{is even}.
    \end{cases}
\end{equation*} 
 Let $v: = (v_1,v_2,\dotsc,v_{n-1})$ and $V:=\cir(v)$. 
If $k \in \{1,2,\dotsc,m-1\}$, then 
each column of $C_k$ has  exactly two ones and remaining entries equal to zero. Further, $k+1$ is odd if and only if $n-k$ is odd. On the other hand, each column of $C_m$ has exactly one entry equal to one and remaining entries equal to zero. Further, the first column of $C_m$ has one 
in the even position if and only if $m$ is even. Also, each $C_k$ is a Toeplitz matrix. In view of these observations, we get

\begin{equation} \label{vCk}
vC_k=\begin{cases}
        (-1)^{k} 2 v &~\mbox{if} ~k=1,\dotsc,m-1 \\
       (-1)^{m}v &~\mbox{if}~k=m.
        \end{cases}
   \end{equation}  
\end{itemize}

\section{Special Laplacian for $W_n$}
We now associate a special Laplacian $\L$ to $W_n$. This definition is motivated from numerical computations.
\begin{definition} \label{lap} \rm
 For each $k \in \{1,2,\dotsc,m\}$, define
 \[g(k):=\frac{n+(-1)^{(m-k)}}{2}\]
 and 
 \begin{equation} \label{al}
\al_k: = \frac{(-1)^{g(k)}(2m^2-6(m-k)^2+1)}{6(n-1)}.
\end{equation}
We say that the $n \times n$ matrix $\L$ defined by
\begin{equation*}
\begin{aligned}
    \L &:= \left[{\begin{array}{cc}
        \frac{n-1}{2} &  \0 \\
        \0 & O
    \end{array}}\right] + \frac{n(n-2)}{6(n-1)}\left[{\begin{array}{cc}
        0 & \0  \\
        \0 & I_{n-1}
    \end{array}}\right] - \frac{1}{2} \left[{\begin{array}{cc}
        0 & \1'  \\
        \1 & O
    \end{array}}\right] + \sum_{k=1}^{m} \al_k  \left[{\begin{array}{cc}
        0 & \0  \\
        \0 & C_k
    \end{array}}\right]
\end{aligned}
\end{equation*}
is the special Laplacian of $W_n$. 
 
\end{definition}

In the rest of the paper, we reserve the notation $\al_1,\dotsc,\al_m$ for the numbers obtained by substituting $k=1,\dotsc,m$ respectively
in the right hand side of the equation (\ref{al}).

\subsection{Illustration for $W_5$ and $W_7$} \label{sec3}
The interconnection between the special Laplacian and the distance matrix for $W_5$ and $W_7$ is given now. Later, in our main result, we  generalize the result mentioned here to a general $n$.
\begin{itemize}
\item Consider $W_5$.  The special vectors are now $c^1=(0,1,0,1)$ and $c^2=(0,0,1,0)$ and the special matrices 
are given by $\cir(c^1)$ and $\cir(c^2)$. We have $\al_1=\frac{1}{8}$ and $\al_2=-\frac{3}{8}$.
The special Laplacian for $W_5$ can now be written easily using the definition:
\begin{equation*}
\begin{aligned}
    \L = \frac{1}{8}\left[\begin{array}{ccccc}
        ~~16 & -4 & -4 & -4 & -4 \\
-4 & ~~5 & ~~1 & -3 & ~~1 \\
-4 & ~~1 & ~~5 & ~~1 & -3 \\
-4 & -3 & ~~1 & ~~5 & ~~1 \\
-4 & ~~1 & -3 & ~~1 & ~~5
\end{array}\right].
\end{aligned}
\end{equation*}
The distance matrix $D$ of $W_5$ is given by
\begin{equation*} 
      D=\left[
{\begin{array}{rrrrrrr}
0 & 1 & 1 & 1 & 1 \\
1 & 0 & 1 & 2 & 1 \\
1 & 1 & 0 & 1 & 2 \\
1 & 2 & 1 & 0 & 1 \\
1 & 1 & 2 & 1 & 0
\end{array}}
\right].
\end{equation*}
By setting $w:=\frac{1}{4}(0,1,1,1,1)'$, we note that 
\begin{equation} \label{w5}
-\frac{1}{2} \L+\frac{4}{n-1}ww'= \frac{1}{4}\left[\begin{array}{ccccc}
   -4 & ~~1 & ~~1 & ~~1 & ~~1 \\
~~1 & -1 & ~~0 & ~~1 & ~~0 \\
~~1 & ~~0 & -1 & ~~0 & ~~1 \\
~~1 & ~~1 & ~~0 & -1 & ~~0 \\
~~1 & ~~0 & ~~1 & ~~0 & -1
\end{array}\right].
\end{equation}
The Moore-Penrose inverse of $D$ and the matrix in the right hand side of $(\ref{w5})$ are equal. This can be verified directly.

\item Consider $W_7$. There are three special vectors now. These are given by 
\[c^1=(0,1,0,0,0,1),~~c^2=(0,0,1,0,1,0),~\mbox{and}~~c^3=(0,0,0,1,0,0). \] Using the special matrices $C_1=\cir(c^1)$, $C_2=\cir(c^2)$ and $C_3=\cir(c^3)$ and the numbers $\al_1=-\frac{5}{36}$, $\al_2=-\frac{13}{36}$ and $\al_3=\frac{19}{36}$, we compute the
special Laplacian for $W_7$:
\begin{equation*}
\begin{aligned}
    \L = \frac{1}{36}\left[\begin{array}{ccccccc}
        ~108 & -18 & -18 & -18 & -18 & -18 & -18 \\
-18 & ~~35 & -~5 & -13 & ~~19 & -13 & -~5 \\
-18 & -~5 & ~~35 & -~5 & -13 & ~~19 & -13 \\
-18 & -13 & -~5 & ~~35 & -~5 & -13 & ~~19 \\
-18 & ~~19 & -13 & -~5 & ~~35 & -~5 & -13 \\
-18 & -13 & ~~19 & -13 & -~5 & ~~35 & -~5 \\
-18 & -~5 & -13 & ~~19 & -13 & -~5 & ~~35
\end{array}\right].
\end{aligned}
\end{equation*}
The distance matrix $D$ of $W_7$ is given by
\begin{equation*}
      D=\left[
{\begin{array}{rrrrrrr}
0 & 1 & 1 & 1 & 1 & 1 & 1 \\
1 & 0 & 1 & 2 & 2 & 2 & 1 \\
1 & 1 & 0 & 1 & 2 & 2 & 2 \\
1 & 2 & 1 & 0 & 1 & 2 & 2 \\
1 & 2 & 2 & 1 & 0 & 1 & 2 \\
1 & 2 & 2 & 2 & 1 & 0 & 1 \\
1 & 1 & 2 & 2 & 2 & 1 & 0
\end{array}}
\right].
\end{equation*}
By setting $w:=\frac{1}{4}(-2,1,1,1,1,1,1)'$, we note that 
$$-\frac{1}{2} \L+\frac{4}{n-1}ww'= \frac{1}{18}\left[\begin{array}{ccccccc}
 -24 & ~~3 & ~~3 & ~~3 & ~~3 & ~~3 & ~~3 \\
~~3 & -8 & ~~2 & ~~4 & -4 & ~~4 & ~~2 \\
~~3 & ~~2 & -8 & ~~2 & ~~4 & -4 & ~~4 \\
~~3 & ~~4 & ~~2 & -8 & ~~2 & ~~4 & -4 \\
~~3 & -4 & ~~4 & ~~2 & -8 & ~~2 & ~~4 \\
~~3 & ~~4 & -4 & ~~4 & ~~2 & -8 & ~~2 \\
~~3 & ~~2 & ~~4 & -4 & ~~4 & ~~2 & -8
\end{array}\right].$$
The matrix on the right hand side of the above equation is the Moore-Penrose inverse of $D$.
\end{itemize}

\section{Main result}
We are now ready to state our main result.
The Moore-Penrose inverse of $D$ is given by
\begin{equation} \label{mpd}
D^{\dag} = -\frac{1}{2} \L+\frac{4}{n-1}ww',
\end{equation}
where $w:=\frac{1}{4}(5-n,1,\dotsc,1)'$. 
 Furthermore, $\L$ has the following properties:
\begin{enumerate}
\item[{\rm (i)}] $\L$ is positive semidefinite.
\item[{\rm (ii)}] $\L1=\0$. That is, all row/column sums of $\L$ are zero.
\item[{\rm (iii)}] $\rank(\L) = n-2.$
\end{enumerate}
In view of Section \ref{sec3}, the result is true for $W_5$ and $W_7$. We now proceed to show that the result holds for any odd integer $n$. In the rest of the paper, we assume $n \geq 9$.

\subsection{Some identities}
To prove the main result, we need the
following identities.
\begin{lemma}
Let $n$ be odd and $m=\frac{n-1}{2}$. Define $$g(k):=\frac{n+(-1)^{m-k}}{2}~~~k =1,2,\dotsc,m.$$
 Then the following are true.
\begin{enumerate}
\item[$\I$] $\sum_{k=1}^{m} (-1)^{g(k)} {(2m^2-6(m-k)^2+1)}=\begin{cases}
        -3m^2+3m  &~\mbox{if}~m~\mbox{is even} \\
        -m^2+3m+1&~\mbox{if}~m~\mbox{is odd.}
         \end{cases}$
\item[$\J$] $\sum_{k=1}^{m}  (-1)^{k+g(k)} {(2m^2-6(m-k)^2+1)} = -3m^2.$

\item[$\K$] $2 \sum_{k=1}^{m} \al_k - \al_m= \dfrac{6m-4m^2+1}{6(n-1)}$.

\item[$\KK$] If $j,j-1, j-2$ belong to $\{1,2\dotsc,m\}$, then $$2 \al_{j-1}+\al_j+\al_{j-2}=(-1)^{j} \frac{2}{n-1}.$$

\item[$\KKK$] \[2\sum_{k=1}^{m} (-1)^{k} \al_k - (-1)^{m} \al_m=\frac{2n-n^2}{6(n-1)}. \]

\end{enumerate}
\end{lemma}
\begin{proof}
We begin with the proof of $\I$.    
%Define $$g(k):=\frac{n+(-1)^{m-k}}{2}.$$

{\it Case $1$}. Suppose $m$ is even. Then,  \[g(k)=
    \begin{cases}
    \frac{n+1}{2} & k ~\mbox{is even} \\
    \frac{n-1}{2} &k  ~\mbox{is odd}.\\
    \end{cases}\]

Since $m+1=\frac{n+1}{2}$ and $m$ is assumed to be even, we have 
\begin{equation} \label{sign1}
(-1)^{g(k)}=
    \begin{cases}
    -1 &k  ~\mbox{is even}\\
    ~~1 & k ~\mbox{is odd}. 
    \end{cases}
    \end{equation}
    
Therefore,
\begin{equation} \label{even1}
\sum_{k=1}^{m} (-1)^{g(k)}=0. 
\end{equation}

We now use the formula: If $p$ is even, then, \[1-2+3-4+5-\cdots-p=-\frac{p}{2}.\]

Applying this to (\ref{sign1}), we get 

\begin{equation} \label{even2}
\sum_{k=1}^{m} (-1)^{g(k)}k=-\frac{m}{2}. 
\end{equation}

If $p$ is even, then we know that \[1^2-2^2+3^2-\cdots-p^2=-\frac{p(p+1)}{2}.\]

By this formula, we deduce
\begin{equation} \label{even3}
\sum_{k=1}^{m}(-1)^{g(k)}k^2=-\frac{m(m+1)}{2}. 
\end{equation}

By $(\ref{even1}), (\ref{even2})$ and $(\ref{even3})$,
\begin{equation*}
\begin{aligned}
\sum_{k=1}^{m} (-1)^{g(k)}{(2m^2-6(m-k)^2+1)}&=-6\sum_{k=1}^{m} (-1)^{g(k)} (m-k)^2 \\
&=12m \sum_{k=1}^{m} (-1)^{g(k)}k-6 \sum_{k=1}^{m} (-1)^{g(k)}k^2 \\
&=12m (\frac{-m}{2})-6(\frac{-m(m+1)}{2}) \\
&=3m-3m^2.
\end{aligned}
\end{equation*}

{\it Case $2$}: Suppose $m$ is odd. 
Then,
\begin{equation} \label{sign2}
(-1)^{g(k)}=
    \begin{cases}
    ~~1 & k ~\mbox{is odd} \\
    -1 &k  ~\mbox{is even}.\\
    \end{cases}
    \end{equation}

Therefore,
\begin{equation} \label{odd1}
\sum_{k=1}^{m} (-1)^{g(k)}=1.
\end{equation}

If $p$ is odd, then, $$1-2+3-4-\cdots+p=\frac{p+1}{2}.$$ 

In view of this formula, we have 
\begin{equation} \label{odd2}
\sum_{k=1}^{m} (-1)^{g(k)} k=\frac{m+1}{2}.
\end{equation}

If $p$ is odd, then \[1-2^2+3^2-4^2+\cdots +p^2=\frac{p(p+1)}{2}.\] 

So,
\begin{equation} \label{odd3}
\sum_{k=1}^{m} (-1)^{g(k)} k^2=\frac{m(m+1)}{2}.
\end{equation}

By $(\ref{odd1}), (\ref{odd2})$, and $(\ref{odd3})$,
\begin{equation} \label{odd11}
\begin{aligned}
\sum_{k=1}^{m} (-1)^{g(k)} (m-k)^{2} &=m^2+\frac{m(m+1)}{2}-m(m+1) \\
&=\frac{m^2-m}{2}.\\
\end{aligned}
\end{equation}

Again by $(\ref{odd1})$, $(\ref{odd2})$ and $(\ref{odd3})$, and by $(\ref{odd11})$, we get
\begin{equation*}
\begin{aligned}
\sum_{k=1}^{m} (-1)^{g(k)}{(2m^2-6(m-k)^2+1)}&=2m^2-6(\frac{m^2-m}{2})+1\\
&=-m^2+3m+1.
\end{aligned}
\end{equation*}
This completes the proof of $\I$.

We now prove $\J$. If $m$ is even, then by (\ref{sign1}), 
$(-1)^{k+g(k)}=-1$ for any $k$. Similarly, if $m$ is odd, then by $(\ref{sign2})$, $ (-1)^{k+g(k)}=-1$ for any $k$.
Thus we have,
\begin{equation*}
    \begin{aligned}
       \sum_{k=1}^{m}  (-1)^{k+g(k)} {(2m^2-6(m-k)^2+1)} &= (-1)\sum_{k=1}^{m}  {(-4m^2-6k^2+12mk+1)} \\ 
       &=(4m^2-1) \sum_{k=1}^{m} 1+6 \sum_{k=1}^{m} k^2 -12m \sum_{k=1}^{m} k\\
        &= -3m^2.
    \end{aligned}
\end{equation*}
The proof of $\J$ is complete.

We now prove $\K$.  
Define $\delta:=2 \sum_{k=1}^{m}\al_k -\al_m$.
Suppose $m$ is even. Then by $\I$,
\[\sum_{k=1}^{m} \al_k=\frac{-3m^2+3m}{6(n-1)}. \] 
By definition,
\[\al_m=(-1)^{\frac{n+1}{2}}\frac{2m^2+1}{6(n-1)}. \]
Since $m=\frac{n-1}{2}$ and $m$ is even, $m+1=\frac{n+1}{2}$ is odd. 
Hence, 
\[\al_m=- \frac{2m^2+1}{6(n-1)}. \]
Now,
 \begin{equation*} \label{del1}
 \begin{aligned}
\delta &=\frac{1}{6(n-1)}(2(3m-3m^2)+2m^2+1) \\
&=\frac{6m-4m^2+1}{6(n-1)}.
\end{aligned}
\end{equation*}
If $m$ is odd, then by $\I$
\[\sum_{k=1}^{m} \al_k=\frac{-m^2+3m+1}{6(n-1)}. \] Also, by definition \[\al_m=\frac{2m^2+1}{6(n-1)}. \]
Now, 
\begin{equation*}\label{del2}
\begin{aligned}
\delta  &= \frac{1}{6(n-1)}
        2(-m^2+3m+1)-(2m^2+1)  \\
        &=\frac{6m-4m^2+1}{6(n-1)}. \\
        \end{aligned}
\end{equation*}
The proof of $\K$ is complete.

\noindent We now prove $\KK$. 
In view of (\ref{al}), 
 \[\al_{j-1}=\frac{(-1)^{\frac{n+(-1)^{(m-(j-1))}}{2}}(2m^2-6(m-(j-1))^2+1)}{6(n-1)},\]
 $$\al_j=\frac{(-1)^{\frac{n+(-1)^{(m-j)}}{2}}(2m^2-6(m-j)^2+1)}{6(n-1)},$$ \[\al_{j-2}=\frac{(-1)^{\frac{n+(-1)^{(m-(j-2))}}{2}}(2m^2-6(m-(j-2))^2+1)}{6(n-1)}.\]
We note that
\begin{equation}\label{aj} 
\begin{aligned}
\al_{j} + \al_{j-2}&= \frac{(-1)^{\frac{n+(-1)^{(m-j)}}{2}}(-8m^2-12j^2+24mj+24j-24m-22)}{6(n-1)} \\ &=  \frac{(-1)^{j-1}(8m^2+12j^2-24mj-24j+24m+22)}{6(n-1)} .
\end{aligned}
\end{equation}
Also,
\begin{equation}\label{aj-2}
\begin{aligned}
2\al_{j-1}&= \frac{(-1)^{\frac{n-(-1)^{(m-j)}}{2}}(-8m^2-12j^2+24j+24mj-24m-10)}{6(n-1)} \\ &= \frac{(-1)^{j}(8m^2+12j^2-24mj-24j+24m+10)}{6(n-1)}. 
\end{aligned}
\end{equation}
Adding equations (\ref{aj}) and (\ref{aj-2}), we get
\begin{equation*} \label{pp}
    \begin{aligned}
    2\al_{j-1}+\al_{j} + \al_{j-2} &= (-1)^j \frac{2}{n-1}.
    \end{aligned}
\end{equation*}
This proves $\KK$. The proof of $\KKK$ is direct from $\J$. 
This completes the proof.
\end{proof}

%\subsection{Outline of the proof}

\section{Computation of $\L D$}
To prove the inverse formula, it is useful to compute $\L D$ precisely. 
\begin{lemma} \label{M}
\[\L D=
\left[\begin{array}{cccc}
\dfrac{1-n}{2} & \dfrac{5-n}{2} \1'  \\
        \\
        \dfrac{1}{2} \1 & M
\end{array}\right],\]
where
\[M:=\cir (\frac{n(n-2)}{6(n-1)}u-\frac{1}{2}\1'+\sum_{k=1}^{m} \al_k c^{k} \D).\]
\end{lemma}

\begin{proof}
Direct multiplication of $\L$ and $D$ gives
\[\L D=
\left[\begin{array}{cccc}
\dfrac{1-n}{2} & A\\
\\
B & M
\end{array}\right],\]
where 
\[A= \frac{n-1}{2} \1'-\frac{1}{2}\1'\D,\]
\[B= \frac{n(n-2)}{6(n-1)}\1+\sum_{k=1}^{m}  \al_k C_k \1,\]
\[M=\frac{n(n-2)}{6(n-1)}\D-\frac{1}{2}\1 \1'+\sum_{k=1}^{m} \al_k  C_k\widetilde{D}.\]
We now simplify $A$, $B$ and $M$.
Since $\D=\cir (u )$, $\1 \1'=\cir(\1')$ and $C_k\D=\cir({c^k}\D)$, we get
\[M=\cir (\frac{n(n-2)}{6(n-1)}u-\frac{1}{2}\1'+\sum_{k=1}^{m} \al_k {c^{k}} \D).\]
To complete the proof, we need to show that 
\[A=\frac{5-n}{2} \1'~~\mbox{and}~~B=\frac{1}{2} \1. \]
By (\ref{fd}), 
     $\D \1 = {2}{(n-3)} \1$.
     So, $A=\frac{n-1}{2} \1'-(n-3) \1' $.
%\label{E_A}
This gives
\begin{equation*} \label{E_A}
A=\frac{5-n}{2}\1'.
\end{equation*}
To simplify $B$, we make the following observation first. If $1 \leq k \leq m-1$, then the the first row of $C_k$ has  exactly two ones and remaining entries equal to zero. 
On the other hand, the first row of $C_m$ has  exactly one entry equal to one and remaining entries equal to zero. 
Using this observation together with the fact that  $C_k$ is circulant, we now get
\[C_k \1=
    \begin{cases}
    2 \1 & \mbox{if~}k=1,\dotsc,m-1\\
    \1 & \mbox{if~}k=m.
    \end{cases}\]
So,
\begin{equation} \label{B1}
\begin{aligned}
B&=\frac{n(n-2)}{6(n-1)}\1+2\sum_{k=1}^{m-1}  \al_k \1+\al_m \1 \\
&=\frac{n(n-2)}{6(n-1)}\1 + 2 \sum_{k=1}^{m} \al_k \1-2 \al_m \1 + \al_m \1 \\
&=\frac{n(n-2)}{6(n-1)}\1+2 \sum_{k=1}^{m} \al_k \1 -\al_m \1.\\
\end{aligned}
\end{equation} 
%We know that $m=\frac{n-1}{2}$ and $n$ is odd. Hence, $m$ is even if and only if $\frac{n+1}{2}$ is odd.
%This implies
%\[\al_m=
 %  \begin{cases}
  %  - \frac{2m^2+1}{6(n-1)} & \mbox{if~} m ~\mbox{is even}\\
   % \\
    %\frac{2m^2+1}{6(n-1)} & \mbox{if~}m~\mbox{is odd}.
    %\end{cases}\]
Let $$\delta=2 \sum_{k=1}^{m} \al_k-\al_m.$$
 Then by $\K$, 
 \[ \delta =\frac{6m-4m^2+1}{6(n-1)}.\]
Hence ($\ref{B1}$) reduces to
\[
B=\frac{n(n-2)+6m-4m^2+1}{6(n-1)} \1.
 \]
Since $m=\frac{n-1}{2}$, \[n(n-2)-4m^2+6m+1=3n-3. \]
 So, \[ B=\frac{1}{2} \1.\]
The proof is complete now. 
\end{proof}

\subsection{The vectors $c^k\D$}
We now compute the vectors $c^1\D,\dotsc,c^m\D$ which appear in the matrix $M$. 
\begin{lemma} \label{first_row}
$c^{1} \widetilde{D}=({2,2},3,\underbrace{4,\dotsc,4}_{n-6},3,2)$
and $c^{1}\D \in \Delta$.
\end{lemma}
\begin{proof}
We first note that \[c^1=(0,1,0,\dotsc,0,1).\] So, $c^{1}\widetilde{D}$ is the sum of the second row and the last
row of $\widetilde{D}$. Let $x$ be the second row and $y$ be the last row of $\widetilde{D}$. Then,
\[x=(1,0,1,2,\dotsc,2)~~\mbox{and}~~y=(1,2,\dotsc,2,1,0).\]
Now, $c^1\D=x+y=({2,2},3,\underbrace{4,\dotsc,4}_{n-6},3,2)$. 
To verify $c^1\D \in \Delta$ is direct.
This completes the proof.
\end{proof}

\begin{lemma} \label{n-1/2_row}
\[c^m\widetilde{D}=(\underbrace{2,\dotsc,2}_{\frac{n-3}{2}},1,0,1,2,\dotsc,2), \]
and $c^m\D \in \Delta$.
\end{lemma}
\begin{proof}
We write $c^m$:
\[c^m=(\underbrace{0,\dotsc,0}_{\frac{n-1}{2}},1,{0,\dotsc,0}).\] Put $j=\frac{n+1}{2}$.
Then, $c^m\widetilde{D}$ is the $j^{\rm th}$ row of $\widetilde{D}$. 
This means that if $(r_1,\dotsc,r_n)$ is the $(j+1)^{\rm th}$  row of $D$, then 
$c^m\D=(r_{2},\dotsc,r_n)$. The vertex $w_{j+1}$ in $W_n$ is adjacent to $w_1$, $w_j$ and $w_{j+2}$.  Thus, 
\[r_\nu=
    \begin{cases}
    0 & \mbox{if~}\nu=j+1 \\
    1 & \mbox{if~}\nu=1,j, j+2\\
    2 & \mbox{otherwise.}
    \end{cases}\]
If $(\theta_{1},\dotsc,\theta_{n-1})=(r_{2},\dotsc,r_n)$, then the above equation gives
\[\theta_{i}=
    \begin{cases}
    0 & \mbox{if~}i=j \\
    1 & \mbox{if~}i=j-1, j+1\\
    2 & \mbox{otherwise.}
    \end{cases}\]
As $j=\frac{n+1}{2}$ and $c_m\D=(\theta_1,\dotsc,\theta_{n-1})$, we conclude that
\[c^m\widetilde{D}=(\underbrace{2,\dotsc,2}_{\frac{n-3}{2}},1,0,1,2,\dotsc,2).\]
Again, $c^m\D \in \Delta$ is direct. This completes the proof.
\end{proof}

\begin{lemma} \label{n-3/2_row}
\[c^{m-1}\widetilde{D}=(\underbrace{4,\dotsc,4}_{\frac{n-5}{2}},3,2,2,2,3,\underbrace{4,\dotsc,4}_{\frac{n-7}{2}}), \]
and $c^{m-1}\D \in \Delta$.
\end{lemma}
\begin{proof}
Since,
\[c^{m-1}=(\underbrace{0,\dotsc,0}_{\frac{n-3}{2}},1,0,1,\underbrace{0,\dotsc,0}_{\frac{n-5}{2}}),\]
$c^{m-1}\D$ is the sum of $(\frac{n-1}{2})^{\rm th}$ and $(\frac{n+3}{2})^{\rm th}$ rows of $\widetilde{D}$. Let these two rows be $\theta:=(\theta_1,\dotsc,\theta_{n-1})$ and $\rho:=(\rho_{1},\dotsc,\rho_{n-1})$
respectively.  

Suppose $(s_1,\dotsc,s_n)$ is the $(\frac{n+1}{2})^{\rm th}$ row of $D$. Then, $(s_2,\dotsc,s_n)$ is the $(\frac{n-1}{2})^{\rm th}$ row of $\D$.
The vertex $w_{\frac{n+1}{2}}$ is adjacent to $w_1$, $w_{\frac{n+3}{2}}$ and $w_{\frac{n-1}{2}}$ in $W_n$. Thus,
\[s_{\nu}=
    \begin{cases}
    0 & \mbox{if~}\nu=\frac{n+1}{2}\\
    1 & \mbox{if~}\nu=1,\frac{n+3}{2},\frac{n-1}{2} \\
    2 & \mbox{otherwise.}
    \end{cases}\]
As $(\theta_1,\dotsc,\theta_{n-1})=(s_2,\dotsc,s_n)$,
\begin{equation} \label{th}
\theta_{i}=
    \begin{cases}
    0 & \mbox{if~}i=\frac{n-1}{2}\\
    1 & \mbox{if~}i=\frac{n+1}{2},\frac{n-3}{2} \\
    2 & \mbox{otherwise.}
    \end{cases}
    \end{equation}
Suppose $(t_{1},\dotsc,t_n)$ is the $(\frac{n+5}{2})^{\rm th}$ row of $D$. Then, $(t_{2},\dotsc,t_n)$ is the $(\frac{n+3}{2})^{\rm th}$ row of $\D$.
The vertex $w_\frac{n+5}{2}$ is adjacent to $w_1$, $w_\frac{n+3}{2}$ and $w_\frac{n+7}{2}$.
Thus,
\[t_{\nu}=
    \begin{cases}
    0 & \mbox{if~}\nu=\frac{n+5}{2}\\
    1 & \mbox{if~}\nu=1,\frac{n+3}{2},\frac{n+7}{2} \\
    2 & \mbox{otherwise.}
    \end{cases}\]
As $(\rho_1,\dotsc,\rho_{n-1})=(t_2,\dotsc,t_n)$,
\begin{equation} \label{rho}
\rho_{i}=
    \begin{cases}
    0 & \mbox{if~}i=\frac{n+3}{2}\\
    1 & \mbox{if~}i=\frac{n+1}{2},\frac{n+5}{2} \\
    2 & \mbox{otherwise.}
    \end{cases}
    \end{equation}
 By $(\ref{th})$ and $(\ref{rho})$, \[(c^{m-1}\D)_{i}=(\theta+\rho)_{i}= \begin{cases}
2& \mbox{for}~i={\frac{n-1}{2}},{\frac{n+1}{2}}, \frac{n+3}{2}\\
3&\mbox{for}~i={\frac{n-3}{2}},{\frac{n+5}{2}}\\
4&\mbox{otherwise}.
\end{cases}\]
We now show that $c^{m-1}\D \in \Delta$. 
Define $$\Omega_{1}:=\{\frac{n-1}{2},\frac{n+1}{2},\frac{n+3}{2}\},$$
$$\Omega_{2}:=\{\frac{n-3}{2},\frac{n-5}{2} \},$$
\[\Omega_3:=\{2,\dotsc,n-1\} \smallsetminus (\Omega_1 \cup \Omega_2).\]
It is easy to see that, for each $j=1,2,3$,
\[\nu \in \Omega_j~~\iff ~~n+1-\nu \in \Omega_{j},\]
and hence $c^{m-1}\D \in \Delta$.
The proof is complete.
\end{proof}

\begin{lemma}\label{1<k}
Let $1<k<m-1$. Define $q^k:={c^k}\widetilde{D}$.  If $q^k=(q_{1}^k,\dotsc,q_{n-1}^k)$, then
\[q_j^k=
    \begin{cases}
    2 & \mbox{if~}j=k+1,n-k \\
    3 & \mbox{if~}j=k,k+2,n-k-1,n-k+1\\
    4 & \mbox{otherwise.}
    \end{cases}\]
Furthermore, each $q^k \in \Delta$.
\end{lemma}
\begin{proof}
Let $1<k<m-1$. Since $c^k$ has $1$ in the $(k+1)^{\rm th}$ and $(n-k)^{\rm th}$ positions and zeros elsewhere, $q^k$ is the sum of $(k+1)^{\rm th}$ and $(n-k)^{\rm th}$ rows of $\widetilde{D}$. Let these rows be $\theta=(\theta_1,\dotsc,\theta_{n-1})$ and $\eta=(\eta_1,\dotsc,\eta_{n-1})$ respectively.
Let $(s_1,\dotsc,s_n)$ be the $(k+2)^{\rm th}$ row of $D$. The vertex $w_{k+2}$ is adjacent to $w_1$, $w_{k+1}$ and $w_{k+3}$.

We now have
\begin{equation*}
s_{j}=
    \begin{cases}
    0 & \mbox{if~}j=k+2 \\
    1 & \mbox{if~}j=1,k+1,k+3\\
    2 & \mbox{otherwise.}
    \end{cases}
    \end{equation*}
As $(\theta_1,\dotsc,\theta_{n-1})=(s_2,\dotsc,s_n)$,
\begin{equation} \label{theta}
\theta_j=
    \begin{cases}
    0 & \mbox{if~}j=k+1 \\
    1 & \mbox{if~}j=k,k+2\\
    2 & \mbox{otherwise.}
    \end{cases}
    \end{equation}
Let $(n-k+1)^{\rm th}$ row of $D$ be $(t_1,\dotsc,t_n)$.  Then,
\begin{equation*} \label{t1}
t_{j}=
    \begin{cases}
    0 & \mbox{if~}j=n-k+1 \\
    1 & \mbox{if~}j=1,n-k,n-k+2\\
    2 & \mbox{otherwise.}
    \end{cases}
    \end{equation*}
Because $(\eta_1,\dotsc,\eta_{n-1})=(t_{2},\dotsc,t_n)$,
\begin{equation} \label{eta}
\eta_{j}=
    \begin{cases}
    0 & \mbox{if~}j=n-k\\
    1 & \mbox{if~}j=n-k-1,n-k+1\\
    2 & \mbox{otherwise.}
    \end{cases}
    \end{equation}
We now compute $\theta+ \eta$.
Since $1<k<m-1$, we have 
$n-2k > n-2m+1$. As $n-2m+1=2$, $n-2k> 2$. Thus, $n-k-1 > k+1$. Combining this inequality with the fact that $k < m-1$, we have 
$$k+2 \leq n-k-2.$$
From $(\ref{theta})$ and $(\ref{eta})$, we immediately get  
\begin{equation} \label{s1}
(\theta+\eta)_{j}=
    \begin{cases}
    4 & \mbox{if~}j=1,\dotsc,k-1 \\
    3 & \mbox{if~}j=k,k+2\\
    2 & \mbox{if~} j=k+1.
    \end{cases}
    \end{equation}
 If $k+2<j\leq n-k-2$, then $\theta_{j}=\eta_{j}=2$. So,
 \begin{equation}\label{snew}
(\theta+\eta)_{j}=4~~\mbox{for all}~~k+2<j\leq n-k-2.
 \end{equation}
We note that  
\begin{equation} \label{seta}
\eta_j=
    \begin{cases}
    1 & \mbox{if~}j=n-k-1,n-k+1 \\
    0 & \mbox{if~}j=n-k.
    \end{cases}
    \end{equation}
Since $\theta_{j}=2$ for all $j>k+2$ and $n-k-1>k+2 $, we have 
\begin{equation} \label{stheta}
\theta_{j}=2~~ \mbox{for all} ~~j \geq n-k-1.
\end{equation}
In view of $(\ref{seta})$ and $(\ref{stheta})$, 
\begin{equation} \label{s2}
(\theta+\eta)_j=
    \begin{cases}
    3 & \mbox{if~}j=n-k-1,n-k+1 \\
    2 & \mbox{if~}j=n-k.
    \end{cases}
    \end{equation}
Finally, from $(\ref{theta})$ and $(\ref{eta})$, 
\[\theta_{j}=2~~\mbox{and}~~\eta_{j}=2~~~\mbox{for all}~ j> n-k+1.\] So,
\begin{equation} \label{s3}
(\theta+\eta)_{j}=4~~\mbox{if~} n-k+1<j \leq n-1.
\end{equation}
By $(\ref{s1})$, $(\ref{snew})$, $(\ref{s2})$ and $(\ref{s3})$, we get \[q_j^k=
    \begin{cases}
    2 & \mbox{if~}j=k+1,n-k \\
    3 & \mbox{if~}j=k,k+2,n-k-1,n-k+1\\
    4 & \mbox{otherwise.}
    \end{cases}\]
We now show that $q^{k} \in \Delta$.
For this, we partition the set $\Omega:=\{2,\dotsc,n-1\}$ into three parts.
Define $\Omega_1:=\{k+1,n-k\}$, $\Omega_2:=\{k,k+2,n-k-1,n-k+1\}$ and
$\Omega_3:=\Omega \smallsetminus (\Omega_1 \cup \Omega_2)$.
Each $\Omega_i$ has the property 
\[\nu \in \Omega_{i} ~~\iff~~n+1-\nu \in \Omega_i.\]
 Therefore, $q^k \in \Delta$.
\end{proof}

\subsection{Computation of $\sum_{k=1}^{m} \al_k {c^{k}} \D$}
To simplify $M$, we need to compute the linear combination
$$f:=\sum_{k=1}^{m} \al_k {c^k}\D.$$
For $1\leq k \leq m$, define $q^k:={c^k}\D$. 
We shall write $q^k:=(q_{1}^{k},\dotsc,q_{n-1}^{k})$ and
$f:=(f_{1},\dotsc,f_{n-1})$. Now,
\begin{equation*}
\begin{aligned}
(f_1,\dotsc,f_n) &=\sum_{k=1}^{m} \al_k (q_{1}^k,\dotsc,q_{n-1}^{k}) \\
&=(\sum_{k=1}^{m} \al_k q_{1}^{k},\dotsc,\sum_{k=1}^{m} \al_k q_{n-1}^{k}).
\end{aligned}
\end{equation*}
We now compute $f$ precisely.
\begin{lemma} \label{f1}
\[f_{1}=\frac{3-n}{n-1}.\]
\end{lemma}
\begin{proof}
By Lemma \ref{first_row}, \ref{n-1/2_row}, \ref{n-3/2_row} and $\ref{1<k}$, we have 
\[q_1^k=
    \begin{cases}
    2 & \mbox{if~}k=1,m \\
    4 & \mbox{if~}k=2,\dotsc,m-1.\\
    \end{cases}\]
In view of this,
\begin{equation} \label{xm}
    \begin{aligned}
        f_{1}&=\sum_{k=1}^{m} \al_k q_{1}^{k}   \\
        &= 2 \al_1+2\al_m + 4\sum_{k=2}^{m-1} \al_k \\ 
        &=2 \al_1 + 2 \al_m + 4\sum_{k=1}^{m} \al_k- 4 \al_1 - 4 \al_m \\
        &=  -2\al_1-2\al_m + 4\sum_{k=1}^{m} \al_k.
\end{aligned}
\end{equation} 
Let $\delta=2 \sum_{k=1}^{m} \al_k - \al_m$. By $\K$,
\[\delta=\frac{6m-4m^2+1}{6(n-1)}. \]
Therefore,
\begin{equation}\label{x11}
    \begin{aligned}
    4\sum_{k=1}^{m} \al_k-2\al_m = \frac{-8m^2+12m+2}{6(n-1)}.
    \end{aligned}
\end{equation}
From $(\ref{al})$,
\begin{equation}\label{x21}
    \begin{aligned}
        \alpha_1 = \frac{-4m^2+12m-5}{6(n-1)}.
    \end{aligned}
\end{equation}
Substituting $(\ref{x11})$ and $(\ref{x21})$ in $(\ref{xm})$ gives
\[f_1=\frac{2}{n-1}-\frac{2m}{n-1}. \]
Since $m=\frac{n-1}{2}$,

\[f_1=\frac{3-n}{n-1}.\] 
The proof is complete.
\end{proof}

\begin{lemma} \label{f2}
\[ f_{2}=\frac{-n^2+8n-18}{6(n-1)}.\]
\end{lemma}
\begin{proof}
By Lemma \ref{first_row}, \ref{n-1/2_row}, \ref{n-3/2_row} and $\ref{1<k}$,
\[q_{2}^{k} = \begin{cases}
2 & \mbox{if}~k=1, m\\
3 & \mbox{if}~k=2\\
4 &\mbox{otherwise.}
\end{cases}\]
This gives,      
      \begin{equation*}
    \begin{aligned}
        f_2 &= 2(\al_1+\al_m) + 3 \al_2 
       +4 \sum_{k=3}^{m-1} \al_k\\ &=-2\al_1-2\al_m +4 \sum_{k=1}^{m} \al_k- \al_2 
       \end{aligned}
        \end{equation*}
Put $\delta=2 \sum_{k=1}^{m} \al_k -\al_m$.   By $\K$, $$\delta=\frac{6m-4m^2+1}{6(n-1)}.$$
 Now, 
\begin{equation*} \label{f21} 
 f_2=-2 \al_1+2 \delta -\al_2. 
 \end{equation*}
 We note that    
\[\alpha_1 = \frac{(-4m^2+12m-5)}{6(n-1)}~~\mbox{and}~~\al_2=\frac{4m^2-24m+23}{6(n-1)}. \]
 In view of the above equations,
      \begin{equation*}
    \begin{aligned}
        f_2 &= \frac{-2(-4m^2+12m-5)+2(6m-4m^2+1)-(4m^2-24m+23)}{6(n-1)} \\
        &=\frac{-4m^2+12m-11}{6(n-1)} .\\
       \end{aligned}
        \end{equation*}
As $m=\frac{n-1}{2}$,        
       \[f_2=\frac{-n^2+8n-18}{6(n-1)}.\]
\end{proof}

\begin{lemma} \label{fj}
Let $2<j \leq \frac{n-3}{2}$. Then, $$f_j=\begin{cases}
        \dfrac{-2n^2+10n-18}{6(n-1)} &~\mbox{if}~j~\mbox{is even} \\
        \\
       \dfrac{-2n^2+10n+6}{6(n-1)}&~\mbox{if}~j~\mbox{is odd}.
        \end{cases}$$
\end{lemma}
\begin{proof}
Let $j$ be such that $2 < j \leq \frac{n-3}{2}$.
By Lemma $\ref{first_row}$,  \[q^1_j= \begin{cases}
3&~\mbox{if}~j=3\\
4&~\mbox{otherwise}.
\end{cases}\]
By Lemma \ref{n-1/2_row}, 
\[q_{j}^{m}=2 .\]
In view of Lemma \ref{n-3/2_row}, \[q^{m-1}_j = \begin{cases}
3&~\mbox{if}~j=m-1\\
4&~\mbox{otherwise}.
\end{cases}\]
From Lemma $\ref{1<k}$,
\[1<k<\frac{n-3}{2} \implies q^k_j= \begin{cases}
2&~\mbox{if}~j=k+1\\
3&~\mbox{if}~j=k,k+2\\
4&~\mbox{otherwise}.
\end{cases}\]
Together, all these equations give
\[q^k_j= \begin{cases}
2&~\mbox{if}~j=k+1~\mbox{and}~1<k<\frac{n-3}{2}\\
2&~\mbox{if}~k=\frac{n-1}{2}\\
3&~\mbox{if}~j=3~\mbox{and}~ k=1 \\
3&~\mbox{if}~j= k=\frac{n-3}{2} \\
3&~\mbox{if}~j=k,k+2~\mbox{and}~1<k<\frac{n-3}{2}  \\
4&~\mbox{otherwise}.
\end{cases}\]
Thus,
\begin{equation}\label{qjk}
    q^k_j= \begin{cases}
2&~\mbox{if}~k=j-1,m\\
3&~\mbox{if}~k=j,j-2 \\
4&~\mbox{otherwise}.
\end{cases}
\end{equation}
We need to compute 
\[f_j= \sum_{k=1}^{m} \al_k q_{j}^{k}\]
for $2<j \leq \frac{n-3}{2}$. 
By $\KK$,
\begin{equation*}
    \begin{aligned}
    2\al_{j-1}+\al_{j} + \al_{j-2} &= (-1)^j \frac{2}{n-1}.
    \end{aligned}
\end{equation*}
Define $\Omega:=\{j-2,j-1,j,m\}$. By  (\ref{qjk}), we have
\begin{equation}\label{pp1}
    \begin{aligned}
    f_{j} &= 2\al_{j-1} + 2 \al_m +3\al_{j} + 3\al_{j-2} +4 \sum_{k \notin \Omega}\al_k \\ &=  -(2\al_{j-1}  +\al_{j} +\al_{j-2}) - 2 \al_m +4 \sum_{k=1}^m\al_k \\
    &=-(-1)^{j} \frac{2}{n-1}-2(\al_m-2 \sum_{k=1}^{m} \al_k) \\
    &=-(-1)^{j} \frac{2}{n-1}+\frac{6m-4m^2+1}{3(n-1)}.\\
    \end{aligned}
\end{equation}
where the last two equations follow from $\K$ and $\KK$.
Replacing $m$ by $\frac{n-1}{2}$ in $(\ref{pp1})$, we get
$$f_j=\begin{cases}
        \dfrac{-2n^2+10n-18}{6(n-1)} &~\mbox{if}~j~\mbox{is even} \\
        \\
       \dfrac{-2n^2+10n+6}{6(n-1)}&~\mbox{if}~j~\mbox{is odd}.
        \end{cases}$$
The proof is complete.
\end{proof}

\begin{lemma} \label{fn-1/2}
\[ f_{m}=\begin{cases}
        \dfrac{-2n^2+10n-18}{6(n-1)} &~\mbox{if}~m~\mbox{is even} \\
        \\
       \dfrac{-2n^2+10n+6}{6(n-1)}&~\mbox{if}~m~\mbox{is odd}.
        \end{cases} \]
\end{lemma}
\begin{proof}
In view of Lemma \ref{first_row}, \ref{n-1/2_row}, \ref{n-3/2_row} and $\ref{1<k}$, we have
\begin{equation} \label{fm}
    q_{m}^{k}=
    \begin{cases}
    1 &\mbox{if~}k=m\\
    2 &\mbox{if~}k=m-1\\
    3 & \mbox{if~}k=m-2\\
    4 & \mbox{otherwise.} \\
 \end{cases}
\end{equation}
We need to compute
 \[f_{m}=\sum_{k=1}^{m}\al_k q_m^{k}. \] By $(\ref{fm})$,
\begin{equation*}\label{fn-1}
    \begin{aligned}
    f_{m} &= 4\sum_{k=1}^{m-3}\al_k + 3\al_{m-2}+2\al_{m-1}+\al_{m} \\ 
    &=  4\sum_{k=1}^{m}\al_k -\al_{m-2}-2\al_{m-1}-3\al_{m} \\
    &=2(2 \sum_{k=1}^{m} \al_k -\al_m)-\al_m-\al_{m-2}-2 \al_{m-1}.
    \end{aligned}
\end{equation*}
Define $$\delta:=2(\sum_{k=1}^{m} \al_k - \al_m)~\mbox{and}~\gamma:=\al_m+\al_{m-2}+2 \al_{m-1}. $$
In view of $\K$ and $\KK$,
\begin{equation*}
f_{m}=\frac{6m-4m^2+1}{3(n-1)}-(-1)^{m}\frac{2}{n-1}.
\end{equation*}
Upon substituting $m=\frac{n-1}{2}$,
$$f_{m}=\begin{cases}
        \dfrac{-2n^2+10n-18}{6(n-1)} &~\mbox{if}~m~\mbox{is even} \\
        \\
       \dfrac{-2n^2+10n+6}{6(n-1)}&~\mbox{if}~m~\mbox{is odd}.
        \end{cases}$$
The proof is complete.
\end{proof}

\begin{lemma} \label{fn+1/2}
\[ f_{\frac{n+1}{2}}=\begin{cases}
        \dfrac{-2n^2+10n+6}{6(n-1)} &~\mbox{if}~m~\mbox{is even} \\
        \\
       \dfrac{-2n^2+10n-18}{6(n-1)}&~\mbox{if}~m~\mbox{is odd}.
        \end{cases} \]
\end{lemma}
\begin{proof}
In view of Lemma \ref{first_row}, \ref{n-1/2_row}, \ref{n-3/2_row} and $\ref{1<k}$, we have
\begin{equation*}
    q_{\frac{n+1}{2}}^{k}=
    \begin{cases}
    0 &\mbox{if~}k=m\\
    2 & \mbox{if~}k=m-1\\
    4 & \mbox{otherwise.} \\
 \end{cases}
\end{equation*}
We now have
\begin{equation} \label{fn+1by2}
\begin{aligned}
f_{\frac{n+1}{2}}&=\sum_{k=1}^{m}\al_k q_{\frac{n+1}{2}}^{k} \\
&= 4\sum_{k=1}^{m-2}\al_k + 2\al_{m-1} \\ 
&= 4\sum_{k=1}^{m}\al_k - 2\al_{m-1}-4\al_m.
\end{aligned}
\end{equation}
 By (\ref{al}),
\[\al_m=(-1)^{m+1} \frac{2m^2+1}{6(n-1)}~~\mbox{and}~~\al_{m-1}=(-1)^m \frac{2m^2-5}{6(n-1)}, \]   
    and hence,
\begin{equation}\label{pp3}
   2\al_{m-1}+4\al_m = (-1)^{m-1}\frac{4m^2+14}{6(n-1)}.
    \end{equation}
Suppose $m$ is even. Then by $\I$, (\ref{fn+1by2}) and $(\ref{pp3})$,
\begin{equation*}
\begin{aligned}
f_{\frac{n+1}{2}}&=\frac{4(-3m^2+3m)+4m^2+14}{6(n-1)} \\
&=\frac{-8m^2+12m+14}{6(n-1)}.
    \end{aligned}
    \end{equation*}
     Substituting $m=\frac{n-1}{2}$ gives
    \[f_{\frac{n+1}{2}}=\frac{-2n^2+10n+6}{6(n-1)}. \]
Suppose $m$ is odd. Then by $\I$, (\ref{fn+1by2}) and $(\ref{pp3})$,
\begin{equation*}
\begin{aligned}
f_{\frac{n+1}{2}}&=\frac{4(-m^2+3m+1)-(4m^2+14)}{6(n-1)} \\
&=\frac{-8m^2+12m-10}{6(n-1)}.
    \end{aligned}
    \end{equation*}
    Upon substituting $m=\frac{n-1}{2}$,
    \[f_{\frac{n+1}{2}}=\frac{-2n^2+10n-18}{6(n-1)}. \]    
    This completes the proof.
\end{proof}
\noindent By Lemma $\ref{fj}$, $\ref{fn-1/2}$ and $\ref{fn+1/2}$, we get  
\begin{equation} \label{fjnew} 
2<j \leq \frac{n+1}{2} \implies
f_j=\begin{cases}
        \dfrac{-2n^2+10n-18}{6(n-1)} &~\mbox{if}~j~\mbox{is even} \\
        \\
       \dfrac{-2n^2+10n+6}{6(n-1)}&~\mbox{if}~j~\mbox{is odd.}
        \end{cases}
\end{equation}
To this end, we have computed $f_1,\dotsc,f_{\frac{n+1}{2}}$. We now deduce $f$.
\begin{lemma} \label{f-vector}
Define 
\[f_{1}:=\frac{3-n}{n-1}, \]
\[f_{2}:=\frac{-n^2+8n-18}{6(n-1)}, \]
\[\tau:=\frac{-2n^2+10n-18}{6(n-1)}, \]
\[\omega:=\frac{-2n^2+10n+6}{6(n-1)}. \]
Then,
\[f=(f_1,f_2,\omega,\tau,\omega,\tau,\dotsc,\tau,\omega,f_2). \]
\end{lemma}
\begin{proof}
We begin with the following observation: If $x,y \in \Delta$ and $\beta \in \rr$, then $\beta x+y \in \Delta$.
We have shown that $c^1 \D, c^2\D,\dotsc,c^m\D \in \Delta$. Thus, 
\[f=\sum_{k=1}^{m} \al_k c^k \D \in \Delta.  \] 
So, by Lemma \ref{f1}, Lemma \ref{f2} and equation (\ref{fjnew}), 
\[f=(f_1,f_2,\omega,\tau,\omega,\tau,\dotsc,\tau,\omega,f_2). \]
The proof is complete.
\end{proof}

\subsection{Simplification of $M$}
Using the values of $f_{1}$, $f_2$,  $\omega$ and $\tau$, we simplify
the expression:
 \[M=\cir(\frac{n(n-2)}{6(n-1)}u-\frac{1}{2} \1'+f).  \]

\begin{lemma} \label{msimple}
\[ M=\frac{1}{2}J-2I+\frac{2}{n-1} \cir(1,-1,1,-1,\dotsc,-1). \] 
\end{lemma} 
\begin{proof}
Define \[h:=\frac{n(n-2)}{6(n-1)}u-\frac{1}{2} \1'+f.   \] 
Recall that $u$ is given by $(0,1,2,\dotsc,2,1)$.
Now, \begin{equation} \label{h1}
\begin{aligned}
h_1 &= -\frac{1}{2}+ \frac{3-n}{n-1} \\
&=-\frac{3}{2}+\frac{2}{n-1}.
\end{aligned}
\end{equation}
Suppose $j = 2, n-1$. Since $f_2= f_{n-1}$ and $u_{2}=u_{n-1}=1$, we get
\begin{equation} \label{h2}
\begin{aligned}
h_j &= \frac{n(n-2)}{6(n-1)}-\frac{1}{2}+\frac{-n^2+8n-18}{6(n-1)}\\
       &=\frac{n-3}{n-1}-\frac{1}{2} \\
       &=\frac{1}{2} - \frac{2}{n-1}.
\end{aligned}
\end{equation}
If $2<j<n-1$ is odd, then $f_{j}=\omega$ and $u_j=2$. So,
\begin{equation}\label{hjodd}
\begin{aligned}
h_j&=\frac{2n(n-2)}{6(n-1)}+\frac{6+10n-2n^2}{6(n-1)} -\frac{1}{2}\\
&=\frac{n+1}{n-1} -\frac{1}{2} \\
&=\frac{1}{2}+\frac{2}{n-1}. 
\end{aligned}
\end{equation}
If $2<j<n-1$ is even, then $f_{j}=\tau$ and $u_j=2$. So,
\begin{equation} \label{hjeven}
\begin{aligned}
h_{j}&=\frac{2n(n-2)}{6(n-1)} -\frac{1}{2} + \frac{-2n^2+10n-18}{6(n-1)} \\
&=\frac{n-3}{n-1} -\frac{1}{2} \\
&=\frac{1}{2} - \frac{2}{n-1}.
\end{aligned}
\end{equation}
In view of $(\ref{h1})$, $(\ref{h2})$, $(\ref{hjodd})$ and $(\ref{hjeven})$,
\[h=(-\frac{3}{2}+\frac{2}{n-1}, \frac{1}{2}-\frac{2}{n-1},\frac{1}{2}+\frac{2}{n-1},\dotsc,\frac{1}{2}-\frac{2}{n-1}). \]
Thus, $h$ can be written
\[h=\frac{1}{2}(-3,1,\dotsc,1)+\frac{2}{n-1}(1,-1,1,-1,\dotsc,-1).  \]
It is easy to see that
\[\cir(\frac{1}{2}(-3,1,\dotsc,1))=\frac{1}{2}J-2I. \]
Thus, 
\begin{equation*} \label{finalM}
M=\cir(h)=\frac{1}{2}J-2I+\frac{2}{n-1} \cir(1,-1,1,-1,\dotsc,-1). 
\end{equation*}
\end{proof}

\section{Inverse formula}
We now prove our main result.
\begin{theorem} \label{maint}
Let $W_n$ be a wheel graph with $n$ vertices, where $n$ is an odd integer.
If $D$ is the distance matrix of $W_n$ given by (\ref{E_block_form_D}), then
\[D^{\dag}=-\frac{1}{2} \L +\frac{4}{n-1}ww',  \]
where $w=\frac{1}{4}(5-n,1,\dotsc,1)'$.
\end{theorem}
\begin{proof}
We recall that $v=(1,-1,1,-1,\dotsc,-1)$ and $V=\cir(v)$. 
By Lemma \ref{M}
and Lemma $\ref{msimple}$, we have 
\[\L D=\left[\begin{array}{cccc}
\dfrac{1-n}{2} & \dfrac{5-n}{2} \1'  \\
        \\
        \dfrac{1}{2} \1 & \frac{1}{2}J-2I+\dfrac{2}{n-1}V
\end{array}\right].\]  
We write $\L D$ as
\[\left[\begin{array}{cccc}
\dfrac{5-n}{2}-2 & \dfrac{5-n}{2} \1'  \\
        \\
        \dfrac{1}{2} \1 & \frac{1}{2}\1\1'-2I 
\end{array}\right] + \dfrac{2}{n-1} \left[{\begin{array}{cc}
        0 & \0  \\
        & \\
        \0 & V
    \end{array}}\right].\]  
Thus,
\[\L D= \left[{\begin{array}{cc}
        \dfrac{5-n}{2} & \dfrac{5-n}{2} \1'  \\
        & \\
        \dfrac{1}{2} \1 & \dfrac{1}{2} \1 \1'
    \end{array}}\right] - 2I  + \dfrac{2}{n-1} \left[{\begin{array}{cc}
        0 & \0  \\
        & \\
        \0 & V
    \end{array}}\right].\]  
By an easy verification,
\[2w\1_n'=\left[{\begin{array}{cc}
        \dfrac{5-n}{2} & \dfrac{5-n}{2} \1'  \\
        & \\
        \dfrac{1}{2} \1 & \dfrac{1}{2} \1 \1'
    \end{array}}\right] . \]
Thus,
\begin{equation} \label{ld+2e}
\L D +2 I=2 w \1_n' + \frac{2}{n-1} \left[{\begin{array}{cc}
        0 & \0  \\
        \0 & V
    \end{array}}\right] . 
    \end{equation}
As, \begin{equation*}
    D = \left[ \begin{array}{cc}
        0 & \1' \\
        \1 & \widetilde{D}
    \end{array}\right],
\end{equation*}
where $\widetilde{D}=\cir(u)$, by (\ref{fd}) we deduce,
\begin{equation} \label{dw}
D w=\frac{1}{4} (n-1) \1_n. 
\end{equation}
Define 
\[K:=-\frac{1}{2} \L +\frac{4}{n-1}ww'. \]
To complete the proof, we show that $KD$ is symmetric, $DKD=D$ and $KDK=K$.
We first compute $KD$.
By (\ref{ld+2e}) and (\ref{dw}),
\[\frac{4}{n-1}ww'D=w\1_n', \]
\[-\frac{1}{2} \L D=-w \1_n' - \frac{1}{n-1} \left[{\begin{array}{cc}
        0 & \0  \\
        \0 & V
    \end{array}}\right] + I.\]
Adding the above two equations, we get 
\begin{equation} \label{ad}
KD=I-\frac{1}{n-1} \left[{\begin{array}{cc}
        0 & \0  \\
        \0 & V
    \end{array}}\right] .  
    \end{equation}
So, $KD$ is symmetric.

Before proceeding further, we note that, since $n$ is odd, $uv'=0$ and $\1'v'=0$. Since $\D=\cir(u)$ and $V=\cir(v)$,
$\D V=\cir(uV)$. So, $\D V=O$.

By $(\ref{ad})$, it follows that
 \begin{equation*}\label{dad}
        \begin{aligned}
            DKD &=  D-\frac{1}{n-1}\left[ \begin{array}{cc}
        0 & \1'V \\
        \0 & \widetilde{D}V
    \end{array}\right] \\  &=  D-\frac{1}{n-1}\left[ \begin{array}{cc}
        0 & \0 \\
        \0 & \widetilde{D}V
    \end{array}\right] \\
    &=D.
        \end{aligned}
         \end{equation*}
We now compute $KDK$.  From ($\ref{vCk}$), we recall the following observation for the special matrices $C_1,\dotsc,C_m$ for $W_n$:
\begin{equation} \label{vCkcopy}
vC_k=\begin{cases}
        (-1)^{k} 2 v &~\mbox{if} ~k=1,\dotsc,m-1 \\
       (-1)^{m}v &~\mbox{if}~k=m.
        \end{cases}
   \end{equation}
We claim that  \begin{equation}\label{vl}
    \left[{\begin{array}{cc}
        0 & \0  \\
        \0 & V
    \end{array}}\right]\L=O.
    \end{equation}
By a direct computation, we have    
    \begin{equation*}
    \begin{aligned}
       \left[{\begin{array}{cc}
        0 & \0  \\
        \0 & V
    \end{array}}\right]\L &= \left[{\begin{array}{cc}
        0 & \0  \\
        \0 & N
    \end{array}}\right],
    \end{aligned}
\end{equation*}
where \[N:=\frac{n(n-2)}{6(n-1)}V
    + \sum_{k=1}^{m} \alpha_kVC_k.\]    
    Since $V = \cir(v)$ and $VC_k = \cir(v C_k)$, 
\[N=\cir(\frac{n(n-2)}{6(n-1)}v
    + \sum_{k=1}^{m} \alpha_k vC_k).\]
Using (\ref{vCkcopy}),
\begin{equation*}
\begin{aligned}
\sum_{k=1}^{m} \al_k v C_k &= 2 \sum_{k=1}^{m-1} (-1)^{k} \al_k v + (-1)^{m} \al_m v \\
&=(2 \sum_{k=1}^{m} (-1)^{k} \al_k - (-1)^{m} \al_m) v .\\
\end{aligned}
\end{equation*}    
In view of $\KKK$, 
\[\sum_{k=1}^{m} \al_k v C_k= \frac{2n-n^2}{6(n-1)}v. \]    
    This implies $N=O$ and the claim is proved.
    
    As $\1'v'=0$, we have $V \1=\0$. So, 
\begin{equation} \label{zvw}    
    \left[{\begin{array}{cc}
        0 & \0  \\
        \0 & V
    \end{array}}\right] w=\0.
    \end{equation}
    From equation (\ref{ad}), we obtain
    \[KDK= K+ \frac{1}{2(n-1)}\left[{\begin{array}{cc}
        0 & \0  \\
        \0 & V
    \end{array}}\right] \L-\frac{4}{(n-1)^2} \left[{\begin{array}{cc}
        0 & \0  \\
        \0 & V
    \end{array}}\right] ww'.  \]
By (\ref{vl}) and (\ref{zvw}), we get $KDK=K$.
    The proof is complete.
    \end{proof}

\section{Properties of the special Laplacian matrix}
In this section, we obtain certain properties of the special Laplacian matrix. In order to do this, we need a preliminary result. Define $p:=(p_1,\dotsc,p_{n-2})$ and $q:=(q_1,\dotsc,q_{n-2})$ by 
\[ 
p_k:=\begin{cases}
        -1 &~\mbox{if} ~k=1\\
       -2 &~\mbox{if} ~k~\mbox{is even} \\
       ~~0 & \mbox{else};
        \end{cases}
   \]
   
   \[
q_k:=\begin{cases}
        -1 &~\mbox{if} ~k=1\\
       ~~0 &~\mbox{if} ~k~\mbox{is even} \\
       -2 & \mbox{else}.
        \end{cases}
   \]
Define an $n\times (n-2)$ matrix by
\[C:=\left[{\begin{array}{ccccccccccccccccc}
       2 I_{n-2} \\
       p \\
       q \\
    \end{array}}\right]. \]
%It can be noted that every column of $A$ is an element in $\1^{\perp}$. 
We shall find a matrix 
$X$ such that $\L D X=C$. Define a vector $y:=(y_1,\dotsc,y_{n-3})$ by
\[y_k:=\begin{cases}
        -2 &~\mbox{if} ~k=1\\
       ~~0 &~\mbox{if} ~k~\mbox{is even} \\
       -1 & \mbox{else};
        \end{cases}\]
        and let $Y:=\cir(y)$.

\begin{lemma}\label{L,LDU}
 If \[X:= \frac{1}{2}\left[\begin{array}{cc}
       ~~~{n-7} & ~~~~({n-5})\1_{n-3}' \\
     -\1_{n-3}& 2Y\\
     \0 & O
\end{array}\right],\]
then $\L D X =C$.
\end{lemma}
\begin{proof}
From (\ref{ld+2e}), we have 
\[\L D+2I = 2w\1_n'+\frac{2}{n-1}\left[\begin{array}{cccc}
0 & \0 \\
\0 & V
\end{array}\right]\] 
So,
\begin{equation}\label{LDU}
    \begin{aligned}
        \L D X = 2w\1_n' X-2X+\frac{2}{n-1}\left[\begin{array}{cccc}
0 & \0 \\
\0 & V
\end{array}\right] X.
    \end{aligned}
\end{equation}
By an easy computation, 
\[\1_{n-3}'y = -2-\frac{n-5}{2} = \frac{1-n}{2},\]
and therefore, 
\[\1'_{n-3} Y= \frac{1-n}{2}\1'_{n-3}. \]
Hence
\begin{equation}\label{colsumU}
\begin{aligned}
    \1' X&= \frac{1}{2}\left[\begin{array}{cc}
       {n-7}+(n-3)(-1)  &  ({n-5})\1_{n-3}'+  2 \1_{n-3}' Y
    \end{array}\right] \\
    &= \frac{1}{2}\left[\begin{array}{cc}
       -4  &  ({n-5})\1_{n-3}'+  (1-n)\1_{n-3}'  
    \end{array}\right] \\&= -2\1_{n-2}'.
\end{aligned}
\end{equation}
Recall that $V = \cir(v)$, where $v =(1,-1,1,-1,\dotsc,-1)$ is a row vector with ${n-1}$ components. Define a row vector $\v$ with $n-3$ components by
\[\v:= (1,-1,1,-1,\dotsc,-1).\] 
Let \[R:=[\v',-\v']~\mbox{and}~Q:=\left[\begin{array}{cc}
    ~~1 & -1 \\
    -1 & ~~1
\end{array}\right].\]  Then $V$ can be written
\[V = \left[\begin{array}{cc}
     \cir(\v)& R \\
    R' & Q
\end{array}\right].\]
Therefore
\begin{equation}\label{VU}
    \begin{aligned}
        \left[\begin{array}{cccc}
0 & \0 \\
\0 & V
\end{array}\right] X &=  \frac{1}{2}\left[\begin{array}{ccc}
    0 & ~\0 & ~\0\\
\0& \cir(\v)& R \\
     \0 &     R' & Q
\end{array}\right]  \left[\begin{array}{cc}
    ~~~{n-7} & ~~~~({n-5})\1_{n-3}' \\
     -\1_{n-3}& 2Y\\
     \0 & O
\end{array}\right],\\
&= \frac{1}{2} \left[\begin{array}{ccc}
    0 & ~~\0 \\
     \0& ~2\cir(\v Y) \\
      \0 & ~~~2R'Y
\end{array}\right].
    \end{aligned}
\end{equation}
By a direct verification, we see that
\[\v Y = \frac{1-n}{2}\v.\]
This gives 
\[R'Y = \frac{1-n}{2}R'\]
From (\ref{VU}), we have
\begin{equation}\label{VU1}
    \begin{aligned}
        \left[\begin{array}{cccc}
0 & \0\\
\0 & V
\end{array}\right]X=  \frac{1-n}{2}\left[\begin{array}{ccc}
    0 & ~\0 \\
     \0& \cir(\v) \\
      \0 & R'
\end{array}\right].
    \end{aligned}
\end{equation}
From (\ref{LDU}), (\ref{colsumU}) and (\ref{VU1}), we have
\begin{equation*}
    \begin{aligned}
        \L D X &= -4w\1_{n-2}'-2X-\left[\begin{array}{ccc}
    0 & ~\0 \\
     \0& \cir(\v) \\
      \0 & R'
\end{array}\right] \\
&= -4w\1_{n-2}'-\left[\begin{array}{cc}
    ~~~{n-7} & ~~~~({n-5})\1_{n-3}' \\
     -\1_{n-3}& 2Y\\
     \0 & O
\end{array}\right]
-\left[\begin{array}{ccc}
    0 & ~\0 \\
     \0& \cir(\v) \\
      \0 & R'
\end{array}\right] \\
&= -\left[\begin{array}{cc}
    ~~~~{5-n} & ~~~~~({5-n})\1_{n-3}' \\
    ~~~ \1_{n-3}& J_{n-3}\\
     \1_2 & \1_2\1_{n-3}'
\end{array}\right]-\left[\begin{array}{cc}
    ~~~~{n-7} & ({n-5})\1_{n-3}' \\
     -\1_{n-3}& ~~~2Y+\cir(\v)\\
     \0 & R'
\end{array}\right].
\end{aligned}
\end{equation*}
Thus,
\begin{equation}\label{E,LDU}
    \begin{aligned}
        \L D X &=  -\left[\begin{array}{cc}
    -2 & \0 \\
     ~~\0& ~~~~~~~~~~J_{n-3}+2Y+\cir(\v)\\
     ~~~\1_2 & \1_2\1_{n-3}'+R'
\end{array}\right]\\
&= -\left[\begin{array}{cc}
    -2 & \0 \\
     ~~\0& ~~~~~~~~~~~\cir(\1_{n-3}'+2y'+\v)\\
     ~~~\1_2 & \1_2\1_{n-3}'+R'
\end{array}\right].
\end{aligned}
\end{equation}
We note that
\begin{equation}\label{x}
    \begin{aligned}
        \1_{n-3}'+2y'+\v &= (2,0,2,0,\dotsc,2,0)+2y'\\
        &= -2(1,0,\dotsc,0)'
    \end{aligned}
\end{equation}
and 
\begin{equation}\label{y}
    \begin{aligned}
         \1_2\1_{n-3}'+R'= 2\left[\begin{array}{ccccccc}
             1 & 0 &1 &0 & \dotsc &1 &0 \\
             0 &1 &0 & 1&\dotsc  &0 &1
         \end{array}\right]
    \end{aligned}
\end{equation}
From (\ref{E,LDU}), (\ref{x}) and (\ref{y}), we get
\[\L D X =C.\]
The proof is complete.
\end{proof} 

We conclude the paper with the following theorem.
\begin{theorem}
The special Laplacian matrix $\L$ has the following properties.
\begin{enumerate}
\item[{\rm (i)}] $\L \1_n=\0$. 
\item[{\rm (ii)}] $\rank(\L) = n-2$.
\item[{\rm (iii)}] If $\L^{\dag}=[\theta_{ij}]$, then
 \(d_{ij}=\theta_{ii} + \theta_{jj}-2 \theta_{ij}. \)
  \item[{\rm (iv)}] $\L$ is positive semidefinite.
\end{enumerate}
\end{theorem}

\begin{proof}
Using Definition \ref{lap}, we have\begin{equation*}
\begin{aligned}
    \L\1_n &= \left[{\begin{array}{cc}
        \frac{n-1}{2}\\   \0 
    \end{array}}\right] + \frac{n(n-2)}{6(n-1)}\left[{\begin{array}{cc}
        0 \\ \1
    \end{array}}\right] - \frac{1}{2} \left[{\begin{array}{cc}
        n-1 \\
        \1 
    \end{array}}\right] + \sum_{k=1}^{m} \al_k  \left[{\begin{array}{cc}
        0   \\
        C_k\1
    \end{array}}\right] \\ &= \left[{\begin{array}{cc}
        0\\   -\frac{1}{2}\1+B 
    \end{array}}\right]
\end{aligned}
\end{equation*}
where $B =\frac{n(n-2)}{6(n-1)}\1+\sum_{k=1}^{m}  \al_k C_k \1$. From Lemma \ref{M}, $B = \frac{1}{2}\1$. Hence $\L\1_n = \0$. The proof of (i) is complete.

We will now prove (ii). Since $\L$ is symmetric and $\L \1_n=\0$, all cofactors of $\L$ are equal. Let the common cofactor of $\L$ be $\delta$.
By Theorem $\ref{maint}$, 
\[D^{\dag}=-\frac{1}{2} \L +\frac{4}{n-1}ww'.\] 
Using matrix determinant lemma,
\begin{equation*}
\begin{aligned}
\det( D^{\dag})&= \det (-\frac{1}{2} \L) + \frac{4}{n-1} w' \mbox{adj} (-\frac{1}{2} \L)w \\
&=(-1)^{n-1}\frac{4}{n-1}  \frac{1}{2^{n-1}} \delta.
\end{aligned}
\end{equation*}
Hence $\delta=0$. So, $\rank(\L)\leq n-2.$
In view of Lemma \ref{L,LDU}, $\rank(\L) \geq n-2$.
Thus,
$\rank(\L) = n-2$. This proves (ii).

To prove (iii), we first note that 
\begin{equation} \label{dwlast}
Dw=\frac{n-1}{4} \1_n,~~D^{\dag} \1_n=\frac{4}{n-1}w~~\mbox{and}~~ \1_n'D^\dag \1_n=\frac{4}{n-1}.
\end{equation}
 Define $$P:= I-\frac{1}{n}J~~\mbox{and}~~G:=-\frac{1}{2} PDP.$$ As $\1_n' D^{\dag}\1_n>0$, by Theorem 3.1 in \cite{BALAJIedm}, we have
\begin{equation} \label{ddag}
    D^{\dag}=-\frac{1}{2} G^{\dag} +\frac{1}{\1_n'D^{\dag}\1_n}(D^{\dag}\1_n)(\1_n'D^{\dag}).
\end{equation}
From (\ref{dwlast}) and (\ref{ddag}), we get
\begin{equation} \label{mpinvformula}
D^{\dag}=-\frac{1}{2} G^{\dag} +\frac{4}{n-1}ww'.
\end{equation}
By our inverse formula,
\begin{equation}\label{mpinvformulacopy}
    D^{\dag}=-\frac{1}{2} \L +\frac{4}{n-1}ww'
\end{equation}
Equations (\ref{mpinvformula}) and (\ref{mpinvformulacopy}) imply
\[G^{\dag} = \L.\]
This gives
\begin{equation} \label{glast}
\L^\dag=G=-\frac{1}{2} PDP. 
\end{equation}
By putting $\L^\dag=[\theta_{ij}]$, we see that equation $(\ref{glast})$ gives
\[d_{ij}=\theta_{ii} + \theta_{jj}-2 \theta_{ij}. \]
This proves (iii).

Since $D$ is a Euclidean distance matrix, by a well-known theorem of Schoenberg, $G$ is a positive semidefinite matrix.
Hence, $\L$ is positive semidefinite. This proves (iv).
The proof is complete.
\end{proof}

\bibliography{mybibfile}

\begin{thebibliography}{1}
\expandafter\ifx\csname url\endcsname\relax
  \def\url#1{\texttt{#1}}\fi
\expandafter\ifx\csname urlprefix\endcsname\relax\def\urlprefix{URL }\fi
\expandafter\ifx\csname href\endcsname\relax
  \def\href#1#2{#2} \def\path#1{#1}\fi

\bibitem{must}
M.~Aouchiche, P.~Hansen, Distance spectra of graphs: A survey, Linear algebra
  and its applications 458 (2014) 301--386.
\newblock \href {http://dx.doi.org/10.1016/j.laa.2014.06.010}
  {\path{doi:10.1016/j.laa.2014.06.010}}.

\bibitem{bapat}
R.~B. Bapat, Graphs and matrices, 2nd Edition, Hindustan Book Agency, New
  Delhi, 2018.

\bibitem{fiedler_2011}
M.~Fiedler, Matrices and Graphs in Geometry, Encyclopedia of Mathematics and
  its Applications, Cambridge University Press, 2011.
\newblock \href {http://dx.doi.org/10.1017/CBO9780511973611}
  {\path{doi:10.1017/CBO9780511973611}}.

\bibitem{Graham}
R.~Graham, L.~Lovász, Distance matrix polynomials of trees, Advances in
  Mathematics 29~(1) (1978) 60--88.
\newblock \href {http://dx.doi.org/10.1016/0001-8708(78)90005-1}
  {\path{doi:10.1016/0001-8708(78)90005-1}}.

\bibitem{Balaji}
R.~Balaji, R.~B. Bapat, S.~Goel, An inverse formula for the distance matrix of
  a wheel graph with even number of vertices.
\newblock \href {http://arxiv.org/abs/2006.02841} {\path{arXiv:2006.02841}}.

\bibitem{bapat_kirk}
R.~Bapat, S.~Kirkland, M.~Neumann, On distance matrices and laplacians, Linear
  Algebra and its Applications 401 (2005) 193--209.
\newblock \href {http://dx.doi.org/10.1016/j.laa.2004.05.011}
  {\path{doi:10.1016/j.laa.2004.05.011}}.

\bibitem{sivasu}
R.~Bapat, S.~Sivasubramanian, Inverse of the distance matrix of a block graph,
  Linear and Multilinear Algebra 59~(12) (2011) 1393--1397.
\newblock \href {http://dx.doi.org/10.1080/03081087.2011.557374}
  {\path{doi:10.1080/03081087.2011.557374}}.

\bibitem{Jakli}
G.~Jaklic, J.~Modic, Euclidean graph distance matrices of generalizations of
  the star graph, Applied Mathematics and Computation 230 (2014) 650--663.
\newblock \href {http://dx.doi.org/10.1016/j.amc.2013.12.158}
  {\path{doi:10.1016/j.amc.2013.12.158}}.

\bibitem{BALAJIedm}
R.~Balaji, R.~Bapat, On euclidean distance matrices, Linear Algebra and its
  Applications 424~(1) (2007) 108 -- 117.
\newblock \href {http://dx.doi.org/https://doi.org/10.1016/j.laa.2006.05.013}
  {\path{doi:https://doi.org/10.1016/j.laa.2006.05.013}}.

\end{thebibliography}
\end{document}